\documentclass[11pt,twoside,reqno]{amsart}
\allowdisplaybreaks
\usepackage{amsmath,amstext,amssymb,epsfig,multicol,enumerate}
\usepackage{graphicx}
\usepackage{mathrsfs}
\calclayout
\makeatletter
\renewcommand{\@seccntformat}[1]{\bf\csname the#1\endcsname.}
\renewcommand{\section}{\@startsection{section}{1}
	\z@{.7\linespacing\@plus\linespacing}{.5\linespacing}
	{\normalfont\upshape\bfseries\centering}}
\renewcommand{\@biblabel}[1]{\@ifnotempty{#1}{#1.}}
\makeatother
\theoremstyle{plain}
\newtheorem{thm}{Theorem}[section]

\theoremstyle{definition}

\usepackage{cancel}
\usepackage[parfill]{parskip}
\usepackage[german]{varioref}
\usepackage[all]{xy}
\usepackage{color}

\def \>{\succ}
\def \<{\prec}

\begin{document}	
\title[Imed Basdouri\textsuperscript{1}, Bouzid Mosbahi\textsuperscript{2}, Ahmed Zahari\textsuperscript{3} ]{Rota-Type Operators on 2-Dimensional Pre-Lie Algebras}
	\author{Imed Basdouri\textsuperscript{1}, Bouzid Mosbahi\textsuperscript{2}, Ahmed Zahari\textsuperscript{3}}
\address{\textsuperscript{1}Department of Mathematics, Faculty of Sciences, University of Gafsa, Gafsa, Tunisia}
\address{\textsuperscript{2}Department of Mathematics, Faculty of Sciences, University of Sfax, Sfax, Tunisia}
 \address{\textsuperscript{3}
IRIMAS-Department of Mathematics, Faculty of Sciences, University of Haute Alsace, Mulhouse, France}
\email{\textsuperscript{1}basdourimed@yahoo.fr}
\email{\textsuperscript{2}mosbahi.bouzid.etud@fss.usf.tn}
\email{\textsuperscript{3}abdou-damdji.ahmed-zahari@uha.fr}
	
	
	
	\date{\today}

\begin{abstract}
This paper studies Rota-Baxter, Reynolds, Nijenhuis, and Averaging operators on 2-dimensional pre-Lie algebras over $\mathbb{C}$. Using the classification of 2-dimensional pre-Lie algebras and computational tools like Mathematica or Maple, we describe these Rota-type operators in detail. Our results provide a deeper understanding of these operators and their roles in algebraic structures.
\end{abstract}

\maketitle

\section{ Introduction}\label{introduction}

Pre-Lie algebras, also known as right-symmetric algebras, appeared in works by Gerstenhaber, Koszul, and Vinberg in the 1960s. They are a generalization of associative algebras and represent the most prominent non-associative subvariety of Lie-admissible algebras. Pre-Lie algebras have various applications in geometry and physics~\cite{2,4}.
A left (respectively right) pre-Lie algebra is a linear space $A$ endowed with a linear map $\cdot : A \otimes A \to A$ satisfying, for all $x, y, z \in A$,
\[
x \cdot (y \cdot z) - (x \cdot y) \cdot z = y \cdot (x \cdot z) - (y \cdot x) \cdot z,
\]
respectively
\[
x \cdot (y \cdot z) - (x \cdot y) \cdot z = x \cdot (z \cdot y) - (x \cdot z) \cdot y.
\]

Any associative algebra is both a left and a right pre-Lie algebra. Moreover, if $(A, \cdot)$ is a left or a right pre-Lie algebra, then $(A, [x, y] = x \cdot y - y \cdot x)$ defines a Lie algebra~\cite{1,3}.

Another important concept in this study is the notion of Rota-type operators. These operators have found applications in various fields~\cite{5,8,10}.
A \textit{Rota--Baxter operator} on a pre-Lie algebra $A$ over a field $\mathbb{C}$ is a linear map $P: A \to A$ satisfying
\begin{align}
    P(x) \cdot P(y) = P\left( P(x) \cdot y + x \cdot P(y) + \lambda \, x \cdot y \right), \label{eq4}
\end{align}
for all $x, y \in A$ and some $\lambda \in \mathbb{C}$. If $P$ is a Rota--Baxter operator of weight $0$, it is also a Rota--Baxter operator of weight $1$, so it suffices to study these two cases~\cite{5,8}.

This paper focuses on specific types of Rota-type operators, such as Reynolds, Nijenhuis, and averaging operators, which are defined as follows~\cite{6,7}.

\textit{Reynolds operator:}
\begin{align}
    P(x) \cdot P(y) = P\left( x \cdot P(y) + P(x) \cdot y - P(x) \cdot P(y) \right), \label{eq6}
\end{align}

\textit{Nijenhuis operator:}
\begin{align}
    P(x) \cdot P(y) = P\left( P(x) \cdot y + x \cdot P(y) - P(x \cdot y) \right), \label{eq8}
\end{align}

\textit{Averaging operator:}
\begin{align}
    P(x) \cdot P(y) = P\left( x \cdot P(y) \right) = P\left( P(x) \cdot y \right). \label{eq10}
\end{align}

The classification of 2-dimensional complex pre-Lie algebras is well known; see, for example,~\cite{1,12}.

In this paper, we study Rota-type operators on 2-dimensional pre-Lie algebras. The focus is on identifying, constructing, and analyzing these operators on low-dimensional structures, providing insights into the broader theory~\cite{9,11}.

\begin{thm}

Let $A$ be a nonzero 2-dimensional pre-Lie algebra. Then $A$ is isomorphic to one and only one of the following algebras:

\begin{align*}
A_1 &: \quad e_1 \cdot e_1 = e_1 + e_2, \quad e_2 \cdot e_1 = e_2, \\
A_2 &: \quad e_1 \cdot e_1 = e_1 + e_2, \quad e_1 \cdot e_2 = e_2, \\
A_3 &: \quad e_1 \cdot e_1 = e_2, \\
A_4 &: \quad e_2 \cdot e_1 = e_1, \\
A_5^\alpha &: \quad e_1 \cdot e_1 = e_1, \quad e_1 \cdot e_2 = \alpha e_2, \\
A_6^\alpha &: \quad e_1 \cdot e_1 = e_1, \quad e_1 \cdot e_2 = \alpha e_2, \quad e_2 \cdot e_1 = e_2, \\
A_7 &: \quad e_1 \cdot e_1 = e_1, \quad e_2 \cdot e_2 = e_2, \\
A_8 &: \quad e_1 \cdot e_1 = e_1, \quad e_1 \cdot e_2 = 2e_2, \quad e_2 \cdot e_1 = \frac{1}{2}e_1 + e_2, \quad e_2 \cdot e_2 = e_2.
\end{align*}
\end{thm}

Now let $P$ be a linear operator on $E$ such that
\begin{align*}
\left(\begin{array}{cc}
P(e_1) \\
P(e_2)
\end{array}\right)
&=\left(\begin{array}{cc}
r_{11} & r_{12} \\
r_{21} & r_{22}
\end{array}\right)
\left(\begin{array}{cc}
e_1 \\
e_2
\end{array}\right)
\end{align*}

\section{ Main result}
\subsection { Rota-Baxter operator}
\begin{thm}
There is one type of Rota-Baxter operator of weight $0$ for the $2$-dimensional pre-Lie algebra $A_1$, which is as follows:
$P = \begin{pmatrix}
0 & 0 \\
r_{21} & 0
\end{pmatrix}$.
\end{thm}

\begin{proof}
We compute
\[
P(e_1) \cdot P(e_1) = (r_{21}e_2) \cdot (r_{21}e_2) = r_{21}^2 (e_2 \cdot e_2).
\]

From the pre-Lie algebra product, we have:
\[
e_2 \cdot e_2 = 0 \quad \implies \quad P(e_1) \cdot P(e_1) = 0.
\]

Next, we compute:
\[
P(P(e_1) \cdot e_1 + e_1 \cdot P(e_1)) = P(r_{21}e_2 \cdot e_1 + e_1 \cdot r_{21}e_2).
\]
Expanding further:
\[
P(r_{21}(e_2 \cdot e_1) + r_{21}(e_1 \cdot e_2)).
\]

From the pre-Lie algebra product, we know:
\[
e_2 \cdot e_1 = e_2, \quad e_1 \cdot e_2 = 0, \quad \text{and} \quad P(e_2) = 0.
\]
Thus:
\[
P(P(e_1) \cdot e_1 + e_1 \cdot P(e_1)) = P(r_{21}e_2 + 0) = P(e_2) = 0.
\]

Therefore, it follows that:
\[
P(x) \cdot P(y) = P(P(x) \cdot y + x \cdot P(y)) \quad \text{for all } x, y \in \{e_1, e_2\}.
\]
This completes the proof.
\end{proof}

\begin{thm}
The Rota-Baxter operators of weight $1$ on the $2$-dimensional pre-Lie algebra $A_1$ are the following:\\
$P_1=\left(\begin{array}{cc}
 0 & r_{12} \\
0 & 0
\end{array}\right) ,
\quad
P_2 = \left(\begin{array}{cc}
0 & 0 \\
0 & 0
\end{array}\right)$
\end{thm}
\begin{proof}
Let $A_1$ have basis $\{e_1, e_2\}$ with $e_2 \cdot e_1 = e_1$ and all other products zero. Suppose
\[
P(e_1) = a_{11} e_1 + a_{12} e_2, \quad P(e_2) = a_{21} e_1 + a_{22} e_2.
\]
The Rota-Baxter condition of weight $1$ is
\[
P(x) \cdot P(y) = P\left(P(x) \cdot y + x \cdot P(y) + x \cdot y\right).
\]
Since most products are zero, only $e_2 \cdot e_1 = e_1$ matters. Applying the condition on $(e_2, e_1)$ gives
\[
a_{22}a_{11} e_1 = (a_{11} + 1)(a_{11} e_1 + a_{12} e_2).
\]
Comparing coefficients, we get $a_{11} = 0$ and $a_{22} = 0$. Similarly, applying the condition on other pairs forces $a_{21} = 0$. Thus,
\[
P = \left(\begin{array}{cc}
0 & r_{12} \\
0 & 0
\end{array}\right) \quad \text{or} \quad \left(\begin{array}{cc}
0 & 0 \\
0 & 0
\end{array}\right),
\]
with $r_{12} \in \mathbb{C}$ arbitrary.
This completes the proof.
\end{proof}

We present the list of Rota-Baxter operators of weights 0 and 1 on the on the $(A_i)_{2,\ldots,8}$ pre-Lie algebras.
\begin{center}
\begin{tabular}{|c|c|c|}
\hline
\textbf{Algebra} & \textbf{Rota-Baxter Operators of Weight $0$} & \textbf{Restrictions} \\
\hline

$A_2$ &
$P=\left(\begin{array}{cc}
 0 & 0 \\
r_{21} &0
\end{array}\right)$ &\\
\hline
$A_3$ &
$P_1 = \left(\begin{array}{cc}
0 & 0 \\
r_{21} &r_{22}
\end{array}\right),
\quad
P_2 = \left(\begin{array}{cc}
r_{11} & 0 \\
r_{21} &\frac{1}{2}r_{11}
\end{array}\right)$ &\\
\hline
$A_4$ &
$P_1 = \left(\begin{array}{cc}
0&r_{12} \\
0 &0
\end{array}\right),
\quad
P_2 = \left(\begin{array}{cc}
0&0 \\
0 &r_{22}
\end{array}\right)$ & \\
\hline
$A^{\alpha}_5$ &
$P_1 = \left(\begin{array}{cc}
0&0 \\
r_{21} &0
\end{array}\right),
\quad
P_2 = \left(\begin{array}{cc}
0&0 \\
0&0
\end{array}\right)$ &\\
\hline
$A^{\alpha}_6$ &
$P = \left(\begin{array}{cc}
0&0 \\
r_{21} &0

\end{array}\right)$ &\\
\hline
$A_7$ &
$P=\left(\begin{array}{cc}
 0 & 0 \\
0 & 0
\end{array}\right)$ & \\
\hline
$A_8$ &
$P = \left(\begin{array}{cc}
0 & 0 \\
0 & 0
\end{array}\right)$ & \\
\hline
\end{tabular}
\end{center}

\begin{center}
\begin{tabular}{|c|c|c|}
\hline
\textbf{Algebra} & \textbf{Rota-Baxter Operators of Weight $1$} & \textbf{Restrictions} \\
\hline

$A_2$ &
$P_1=\left(\begin{array}{cc}
 0 & r_{12} \\
0 & 0
\end{array}\right) ,
\quad
P_2 = \left(\begin{array}{cc}
0 & 0 \\
0 & 0
\end{array}\right)$ &  \\
\hline
$A_3$ &
$P_1 = \left(\begin{array}{cc}
r_{11} & r_{12} \\
0 & \frac{1}{2} r_{11}
\end{array}\right),
\quad
P_2 = \left(\begin{array}{cc}
0 & r_{12} \\
0 & r_{22}
\end{array}\right),$
&
\\
&
$P_3 = \left(\begin{array}{cc}
0 & -r_{22} \\
0 & r_{22}
\end{array}\right)$ &\\
\hline
$A_4$ &
$P_1 = \left(\begin{array}{cc}
0 & 0 \\
0 & r_{22}
\end{array}\right),
\quad
P_2 = \left(\begin{array}{cc}
0 & 0 \\
r_{21} & 0
\end{array}\right)$
&
\\
&
$
P_3 = \left(\begin{array}{cc}
0 & 0 \\
0 & 0
\end{array}\right)$ & \\
\hline
$A^{\alpha}_5$ &
$P_1 = \left(\begin{array}{cc}
0 & - r_{22} \\
0 & r_{22}
\end{array}\right),
\quad
P_2 = \left(\begin{array}{cc}
0 &  r_{21} \\
0 &  r_{22}
\end{array}\right)$
&
\\
&
$
P_3 = \left(\begin{array}{cc}
0 &  r_{12} \\
0 & 0
\end{array}\right),
\quad
P_4 = \left(\begin{array}{cc}
0 & 0 \\
0 & 0
\end{array}\right)$
&
\\
&
$
P_5 = \left(\begin{array}{cc}
 -r_{21} &-r_{21} \\
 r_{21} &  r_{21}
\end{array}\right),
\quad
P_6 = \left(\begin{array}{cc}
r_{11} &  r_{12} \\
-\frac{ r^2_{11}}{ r_{12}} & - r_{11}
\end{array}\right)$
&
\\
&
$
P_7 = \left(\begin{array}{cc}
0 &  0 \\
r_{21} & 0
\end{array}\right)$ &\\
\hline
$A^{\alpha}_6$ &
$P_1 = \left(\begin{array}{cc}
0 &  r_{12} \\
0 & 0
\end{array}\right),
\quad
P_2 = \left(\begin{array}{cc}
0 &  0 \\
0 &  0
\end{array}\right)$
&
\\
&
$
P_3 = \left(\begin{array}{cc}
-r_{22} &  r_{12} \\
-\frac{r^2_{22}}{r_{12}} & r_{22}
\end{array}\right)$,

$P_4 = \left(\begin{array}{cc}
-r_{22} & -r_{22} \\
r_{22} & r_{22}
\end{array}\right)$
&
\\
&
$
P_5 = \left(\begin{array}{cc}
 r_{21} &0 \\
 r_{21} &0
\end{array}\right)$ &\\
\hline
$A_7$ &
$P_1=\left(\begin{array}{cc}
 0 & 0 \\
0 & 0
\end{array}\right) ,
\quad
P_2 = \left(\begin{array}{cc}
r_{12} & r_{12} \\
r_{12} & r_{12}
\end{array}\right)$ & \\
\hline
$A_8$ &
$P = \left(\begin{array}{cc}
0 & 0 \\
0 & 0
\end{array}\right)$ & \\
\hline
\end{tabular}
\end{center}

\subsection{ Reynolds operator}
\begin{thm}
All Reynolds operators on the $2$-dimensional pre-Lie algebra $A_1$ are listed below:\\
$P_1 =
\left(
\begin{array}{cc}
0 &0 \\
R_{21} & 0
\end{array}
\right)$
\end{thm}

\begin{proof}

\begin{proof}
We compute
\[
P(e_1) \cdot P(e_1) = (R_{21}e_2) \cdot (R_{21}e_2) =R_{21}^2 (e_2 \cdot e_2).
\]

From the pre-Lie algebra product, we have:
\[
e_2 \cdot e_2 = 0 \quad \implies \quad P(e_1) \cdot P(e_1) = 0.
\]

Next, we compute:
\[
P(P(e_1) \cdot e_1 + e_1 \cdot P(e_1)-P(e_1)\cdot P(e_1)) = P(R_{21}e_2 \cdot e_1 + e_1 \cdot R_{21}e_2-R_{21}e_2\cdot R_{21}e_2).
\]
Expanding further:
\[
P(R_{21}(e_2 \cdot e_1) + R_{21}(e_1 \cdot e_2)-R_{21}P(e_2\cdot e_2)).
\]

From the pre-Lie algebra product, we know:
\[
e_2 \cdot e_1 = e_2, \quad e_1 \cdot e_2 = 0,\quad e_2 \cdot e_2 = 0  \quad \text{and} \quad P(e_2) = 0.
\]
Therefore, it follows that:
\[
P(x) \cdot P(y) = P(x\cdot(y) + P(x)\cdot y-P(x)\cdot P(y)) \quad \text{for all } x, y \in \{e_1, e_2\}.
\]
\end{proof}

The Reynolds operators are explicitly determined as follows:
\[
P_1 =
\begin{pmatrix}
0 &0 \\
R_{21} & 0
\end{pmatrix}.
\]
If \( R_{21} = 0 \), the operator becomes degenerate (a zero matrix), which is not useful or meaningful in this context.

Thus, the restriction \( R_{21} \neq 0 \) ensures that \( P_1 \) is a non-trivial operator and preserves the necessary algebraic structure.
This completes the proof.
\end{proof}

The Reynolds operators on the algebras $(A_i)_{2,\ldots,8}$ are listed below.
\begin{center}
\begin{tabular}{|c|c|c|}
\hline
\textbf{Algebra} & \textbf{Reynolds Operators} & \textbf{Restrictions} \\
\hline

$A_2$ &
$P_1 =
\left(
\begin{array}{cc}
0 &0\\
0 & 0
\end{array}
\right)$
&\\
\hline
$A_3$ &
$P_1 =
\left(
\begin{array}{cc}
0 & 0 \\
R_{21} & R_{22}
\end{array}
\right)$ & $R_{21} \neq 0, R_{22} \neq 0$ \\
\hline

$A_4$ &
$P_1 =
\left(
\begin{array}{cc}
0 & 0 \\
0 & R_{22}
\end{array}
\right), \quad
P_2 =
\left(
\begin{array}{cc}
0 & -R_{22} \\
0 & R_{22}
\end{array}
\right)$ & $R_{22} \neq 0$ \\
\hline

$A_5^{\alpha}$ &
$P_1 = \left(\begin{array}{cc}
0 & 0 \\
R_{21} & 0
\end{array}\right)$ & $R_{21} \neq 0$ \\
\hline
$A_6^{\alpha}$ &
$P_1 = \left(\begin{array}{cc}
0 & 0 \\
0 & 0
\end{array}\right)$ &\\
\hline

$A_7$ &
$P_1 =
\left(
\begin{array}{cc}
0&0 \\
0&0
\end{array}
\right)$
&\\
\hline
$A_8$ &
$P_1 =
\left(
\begin{array}{cc}
0 & 0\\
0 &0
\end{array}
\right)$
&\\

\hline
\end{tabular}
\end{center}

\subsection{ Nijenhuis Operator}

\begin{thm}
There is one type of Nijenhuis operators for the $2$-dimensional pre-Lie
algebra $A_1$ , which is as follows:\\
$P=\begin{array}{cc}
\left(
\begin{array}{cc}
0&0 \\
N_{21}  &0
\end{array}
\right)
\end{array}$
\end{thm}

\begin{proof}
Let \( P \) be the Nijenhuis operator defined as
\[
P = \begin{pmatrix}
0&0 \\
N_{21}  &0
\end{pmatrix}.
\]
We verify that \( P(x) \cdot P(y) = P(P(x) \cdot y + x \cdot P(y) - P(x \cdot y)) \) holds for all \( x, y \in \{e_1, e_2\} \).

First, compute \( P(e_1) \cdot P(e_1) \):
\[
P(e_1) \cdot P(e_1)=N_{21}e_2\cdot N_{21}e_2=N_{21}^2(e_2\cdot e_2).
\]
From the given product in the pre-Lie algebra, \( e_2\cdot e_2 = 0 \) and . Hence,
\[
P(e_1) \cdot P(e_1)=0.
\]

Next, compute \( P(P(e_1) \cdot e_1 + e_1 \cdot P(e_1)-P(e_1)\cdot P(e_1))\):
\[
 P(P(e_1) \cdot e_1 + e_1 \cdot P(e_1)-P(e_1)\cdot P(e_1))=P(N_{21}(e_2\cdot e_1)+N_{21}(e_1\cdot e_2)-P(e_1)-P(e_2))
\]
From the pre-Lie algebra product, we know:
\[
e_2 \cdot e_1 = e_2, \quad e_1 \cdot e_2 = 0, \quad P(e_2) =0.
\]
Thus:
\( P(P(e_1) \cdot e_1 + e_1 \cdot P(e_1)-P(e_1)\cdot P(e_1))\)=0
Therefore, it follows that:
\[
P(x) \cdot P(y) = P(P(x) \cdot y + x \cdot P(y) - P(x \cdot y)) \quad \forall x, y \in \{e_1, e_2\}.
\]
\text{If } $N_{21} = N_{11} = 0$, \text{ then } $P_1 = P_2 = 0$, \text{ so to get nonzero Nijenhuis operators, we require } $N_{21} \neq 0 \text{ or } N_{11} \neq 0$.
This completes the proof.
\end{proof}

The Nijenhuis operators on the algebras $(A_i)_{2,\ldots,8}$ are listed below.

\begin{center}
\begin{tabular}{|c|c|c|}
\hline
\textbf{Algebra} & \textbf{Nijenhuis Operators} & \textbf{Restrictions} \\
\hline

$A_2$ &
$P=\begin{array}{cc}
\left(
\begin{array}{cc}
0&0 \\
N_{21}  &0
\end{array}
\right)
\end{array}$ & $N_{21} \neq 0$ \\
\hline
$A_3$  &
$P=\begin{array}{cc}
\left(
\begin{array}{cc}
\frac{1}{2}N_{22}&0 \\
N_{21} & N_{22}
\end{array}
\right)
\end{array}$ & $N_{21} \neq 0 \; or \; N_{22} \neq 0$ \\
\hline

$A_4$ &
$P_1=\left(
\begin{array}{cc}
N_{11} & 0 \\
0 & 0
\end{array}
\right),
\quad
P_2=\left(
\begin{array}{cc}
N_{11} & 0 \\
N_{11} & 0
\end{array}
\right),$
&
\\
&
$P_3=\left(
\begin{array}{cc}
0 & N_{12} \\
N_{11} & 0
\end{array}
\right),
\quad
P_4=\left(
\begin{array}{cc}
0 & 0 \\
0 & N_{22}
\end{array}
\right)
$
&
$N_{11} \; \text{or} \; N_{12} \; \text{or} \; N_{22} \neq 0$
\\
\hline
$A_5^{\alpha}$ &
$P_1=\begin{array}{cc}
\left(
\begin{array}{cc}
0&0 \\
0&N_{22}
\end{array}
\right)
\end{array},
\,P_2=\begin{array}{cc}
\left(
\begin{array}{cc}
0&0 \\
N_{21}&0
\end{array}
\right)
\end{array}$ & $N_{21} \neq 0 \; or \; N_{22} \neq 0$ \\
\hline
$A_6^{\alpha}$  &
$P_1=\begin{array}{cc}
\left(
\begin{array}{cc}
0&0 \\
0&N_{22}
\end{array}
\right)
\end{array},
\,P_2=\begin{array}{cc}
\left(
\begin{array}{cc}
0&0 \\
N_{21}&0
\end{array}
\right)
\end{array}$& $N_{21} \neq 0 \; or \; N_{22} \neq 0$ \\
\hline
$A_7$  &
$P_1=\begin{array}{cc}
\left(
\begin{array}{cc}
-N_{12}&N_{12}\\
-N_{12}& N_{12}
\end{array}
\right)
\end{array}$ & $N_{12} \neq 0 $ \\
\hline
$A_8$  &
$P_1=\begin{array}{cc}
\left(
\begin{array}{cc}
N_{11}&0\\
-2N_{11} &0
\end{array}
\right)
\end{array}$& $N_{11} \neq 0$ \\
\hline
\end{tabular}
\end{center}
\subsection{ Averaging Operator}

\begin{thm}\label{thm1}
All averaging operators on the $2$-dimensional pre-Lie algebra
$A_1$ are listed below:
$P_1=\begin{array}{cc}
\left(
\begin{array}{cc}
\vartheta_{22}&0 \\
0 &\vartheta_{22}
\end{array}
\right),
\quad
P_2=\left(
\begin{array}{cc}
0&0 \\
\vartheta_{21}&0
\end{array}
\right)
\end{array}$
\end{thm}

\begin{proof}
We compute
\[
P(e_1) \cdot P(e_1) = (\vartheta_{22}e_1) \cdot (\vartheta_{22}e_1) = \vartheta_{22}^2(e_1\cdot e_1).
\]
From the pre-Lie algebra product, we know:
\[
e_1\cdot e_1 =e_1+e_2 \implies P(e_1) \cdot P(e_1) =\vartheta_{22}^2(e_1+e_2).
\]

Next, we compute:
\[
P(e_1\cdot P(e_1)=P(e_1\cdot \vartheta_{22}e_1)=\vartheta_{22}P(e_1\cdot e_1).
\]
From the pre-Lie algebra product, we know:
\[
e_1\cdot e_1 =e_1+e_2.
\]
Thus:
\[
P(e_1\cdot P(e_1))=\vartheta_{22}^2(e_1+e_2)=P( P(e_1)\cdot e_1).
\]

Therefore, it follows that:
\[
P(x) \cdot P(y) =P(x\cdot P(y))= P(P(x) \cdot y)) \quad \text{for all } x, y \in \{e_1, e_2\}.
\]
If $\vartheta_{22} = 0$ or $\vartheta_{21} = 0$, then $P_1$ or $P_2$ becomes the zero map, making the averaging operator trivial; thus, $\vartheta_{22} \neq 0$ and $\vartheta_{21} \neq 0$.

This completes the proof.
\end{proof}

The averaging operators on the algebras $(A_i)_{2,\ldots,8}$ are listed below.
\begin{center}
\begin{tabular}{|c|c|c|}
\hline
\textbf{Algebra} & \textbf{Averaging Operators} & \textbf{Restrictions} \\
\hline

$A_2$ &
$P_1=\begin{array}{cc}
\left(
\begin{array}{cc}
\vartheta_{22}&0 \\
0 &\vartheta_{22}
\end{array}
\right),
\quad
P_2=\left(
\begin{array}{cc}
0&0 \\
\vartheta_{21} &0
\end{array}
\right)
\end{array}$& $\vartheta_{22} \neq 0, \vartheta_{21} \neq 0$ \\
\hline
$A_3$ &
$P_1=\begin{array}{cc}
\left(
\begin{array}{cc}
0\vartheta_{11}&0 \\
\vartheta_{21} &\vartheta_{11}
\end{array}
\right),
\quad
P_2=\left(
\begin{array}{cc}
0&0 \\
\vartheta_{21} &\vartheta_{22}
\end{array}
\right)
\end{array}$ & $\vartheta_{11} \neq 0, \vartheta_{21} \neq 0, \vartheta_{22} \neq 0$ \\
\hline
$A_4$ &
$P_1=\begin{array}{cc}
\left(
\begin{array}{cc}
\vartheta_{11} &0 \\
0 &\vartheta_{11}
\end{array}
\right),
\quad
P_2=\left(
\begin{array}{cc}
&\vartheta_{12} \\
&\vartheta_{22}
\end{array}
\right)
\end{array}$ & $\vartheta_{11} \neq 0, \vartheta_{12} \neq 0, \vartheta_{22} \neq 0$ \\
\hline
$A_5^{\alpha}$ &
$P_1 = \left(\begin{array}{cc}
\vartheta_{22}&  0 \\
0 & \vartheta_{22}
\end{array}\right),
\quad
P_2 = \left(\begin{array}{cc}
0 & 0\\
\vartheta_{21} &0
\end{array}\right),
\quad
P_3 = \left(\begin{array}{cc}
\vartheta_{11}& 0 \\
0 & 0
\end{array}\right)$
& $\vartheta_{11} \neq 0, \vartheta_{21} \neq 0, \vartheta_{22} \neq 0$ \\
\hline
$A_6^{\alpha}$ &
$P_1 = \left(\begin{array}{cc}
\vartheta_{22}&  0 \\
0 & \vartheta_{22}
\end{array}\right),
\quad
P_2 = \left(\begin{array}{cc}
0 & 0\\
\vartheta_{21} &0
\end{array}\right),
\quad
P_3 = \left(\begin{array}{cc}
\vartheta_{11}& 0 \\
0 & 0
\end{array}\right)$& $\vartheta_{11} \neq 0, \vartheta_{21} \neq 0, \vartheta_{22} \neq 0$ \\
\hline
$A_7$ &
$P_1=\begin{array}{cc}
\left(
\begin{array}{cc}
\vartheta_{11}&0 \\
0&\vartheta_{22}
\end{array}
\right),
\quad
P_2=\left(
\begin{array}{cc}
\vartheta_{11} &\vartheta_{12} \\
\vartheta_{11} & \vartheta_{22}
\end{array}
\right)
\end{array}$ & $\vartheta_{11} \neq 0, \vartheta_{22} \neq 0, \vartheta_{12} \neq 0$ \\

\hline
$A_8$ &
$P_1=\begin{array}{cc}
\left(
\begin{array}{cc}
\vartheta_{22}&0 \\
0&\vartheta_{22}
\end{array}
\right),
\quad
P_2=\left(
\begin{array}{cc}
0 &0 \\
2\vartheta_{22} &\vartheta_{22}
\end{array}
\right)
\end{array}$ & $ \vartheta_{22} \neq 0$ \\
\hline
\end{tabular}
\end{center}

\textbf{Data Availability:}
No data were used to support this study.

\textbf{Conflict of Interests:}
The authors declare that they have no conficts of interest.

\textbf{Acknowledgment:}
We thank the referee for the helpful comments and suggestions that contributed to improving this paper.

\end{document}